%% file: main.tex
\documentclass[12pt,reqno]{article}
\usepackage{geometry}
\usepackage{microtype}
\usepackage{graphicx}
\usepackage{amsmath,amsfonts,amssymb}
\usepackage{amsthm, mathtools}
\usepackage{hyperref}
\usepackage[capitalise]{cleveref}
\usepackage{comment}
\usepackage{xcolor}
\usepackage[hyperref,doi,url=false,eprint=false,style=alphabetic,maxbibnames=99]{biblatex}
\usepackage{enumitem}
\usepackage{caption}

\AtEveryBibitem{
	\clearlist{address}
	\clearfield{date}
	\clearfield{isbn}
	\clearfield{issn}
	\clearlist{location}
	\clearfield{month}
	\clearfield{series}
	\clearfield{note}
    \clearfield{urlyear}
    \clearfield{urlmonth}
			 
	\ifentrytype{book}{}{
	\clearlist{publisher}
	\clearname{editor}
			  }
}

\newcommand{\Z}{\mathbb{Z}}
\newcommand{\R}{\mathbb{R}}
\newcommand{\Q}{\mathbb{Q}}

\newcommand{\F}{\mathcal F}

\let\temp\phi
\let\phi\varphi
\let\varphi\temp

\DeclareMathOperator{\interior}{int}

\DeclareMathOperator{\mmod}{mod}

\newtheorem{thm}{Theorem}[section]

\newtheorem{mainthm}{Theorem}

\newtheorem{lem}[thm]{Lemma}
\newtheorem{prop}[thm]{Proposition}
\newtheorem{conj}[thm]{Conjecture}

\theoremstyle{definition}

\newtheorem{constr}[thm]{Construction}
\newtheorem{rmk}[thm]{Remark}
\newtheorem{ques}[thm]{Question}
\newtheorem{fact}[thm]{Fact}

\title{Pseudo-Anosov representatives of stable Hamiltonian structures}
\author{Jonathan Zung\footnote{jzung@mit.edu}}
\date{}

\begin{document}
\maketitle

\begin{abstract}
A pseudo-Anosov homeomorphism of a surface is a canonical representative of its mapping class. In this paper, we explain that a transitive pseudo-Anosov flow is similarly a canonical representative of its stable Hamiltonian class. It follows that there are finitely many pseudo-Anosov flows admitting positive Birkhoff sections on any given rational homology 3-sphere. This result has a purely topological consequence: any 3-manifold can be obtained in at most finitely many ways as $p/q$ surgery on a fibered hyperbolic knot in $S^3$ for a slope $p/q$ satisfying $q\geq 6$, $p\neq 0, \pm 1, \pm 2 \mmod q$. The proof of the main theorem generalizes an argument of Barthelm\'e--Bowden--Mann.

\end{abstract}

\begin{section}{Introduction}
Thurston showed that a pseudo-Anosov homeomorphism of a surface is a canonical representative of its mapping class:
\begin{thm}[\cite{thurston.GeometryDynamicsDiffeomorphisms}]\label{thm:thurston1}
    Suppose $\phi_1$ and $\phi_2$ are pseudo-Anosov homeomorphisms of a closed surface in the same mapping class. Then $\phi_1$ is conjugate to $\phi_2$ by a homeomorphism isotopic to the identity.
\end{thm}

In this paper, we provide an analogue of this statement for pseudo-Anosov flows on 3-manifolds. First, we will explain that a suitable replacement for ``mapping class'' is the notion of ``stable Hamiltonian structure up to stable homotopy''. The reader might not be familiar with stable Hamiltonian structures, so we will introduce them in section 2. For now, it suffices to know that stable Hamiltonian structure is a simultaneous generalization of the notion of the suspension flow of an area preserving diffeomorphism of a surface and the notion of the Reeb flow of a positive or negative contact structure. Marty proved that a special class of Anosov flows, the skew $\R$-covered Anosov flows, are Reeb flows of contact structures \cite{marty.SkewedAnosovFlows}. The following construction extends this symplectic interpretation to other pseudo-Anosov flows:
\begin{constr}\label{constr:main}
    Given a transitive pseudo-Anosov flow $\phi$ on a closed, oriented 3-manifold, one may blow up finitely many orbits to obtain the Reeb flow of a stable Hamiltonian structure $(\omega_\phi,\lambda_\phi)$. Different choices of blowup give rise to stable Hamiltonian structures which are possibly-non-exact stable homotopic.
\end{constr}
\noindent We sometimes abuse terminology and say that $\phi$ is in the stable Hamiltonian class of $(\omega_\phi, \lambda_\phi)$.
\begin{rmk}
The blowups are indeed necessary, even in the Anosov case. A stable Hamiltonian Reeb flow generally has invariant tori separating the positive and negative contact regions, but pseudo-Anosov flows never have invariant tori. On the other hand, the blowups substantially simplify the construction of a smooth invariant volume form compared to \cite{marty.SkewedAnosovFlows} and \cite{asaoka.InvariantVolumesCodimensionone}.
\end{rmk}
\begin{rmk}
There are several competing notions of equivalence between stable Hamiltonian structures. The two we will use are exact stable homotopy and possibly-non-exact stable homotopy. \Cref{constr:main} can be stated in terms of exact stable homotopy, see \cref{constr:main2}. On a first reading, the reader may restrict to the case of a rational homology 3-sphere where the distinction between the two equivalence relations is immaterial.
\end{rmk}

Now we can state the promised analogue of \cref{thm:thurston1}. Note that all results from here on are contingent on the foundations of symplectic field theory.
\begin{mainthm}\label{thm:main}
    Suppose $\phi_1$ and $\phi_2$ are transitive pseudo-Anosov flows on a hyperbolic 3-manifold. If $(\omega_{\phi_1},\lambda_{\phi_1})$ is exact stable homotopic to $(\omega_{\phi_2},\lambda_{\phi_2})$, then $\phi_1$ and $\phi_2$ are orbit equivalent via a homeomorphism isotopic to the identity. When the condition of hyperbolicity is dropped, there are at most finitely many pseudo-Anosov flows in the same stable Hamiltonian class.
\end{mainthm}

Barthelm\'e, Mann, and Bowden proved \cref{thm:main} in the setting of contact structures and skew $\R$-covered Anosov flows \cite{barthelme.mann.OrbitEquivalencesMathbb}. Their argument has two steps. First, they show that it is possible to reconstruct an Anosov flow up to orbit equivalence once one knows the set of free homotopy classes represented by its closed orbits. Second, they appeal to the invariance of cylindrical contact homology which shows that the algebraic count of closed orbits in any free homotopy class does not change during deformations of the contact form. 

The main observation of this paper is that, modulo some technical complications, one may make the following substitutions in their argument:
\begin{align*}
\text{contact structure} &\leadsto \text{stable Hamiltonian structure mod stable homotopy}\\
\text{skew $\R$-covered Anosov flow} &\leadsto \text{pseudo-Anosov flow}\\
\text{cylindrical contact homology} & \leadsto \text{rational symplectic field theory}
\end{align*}
The two issues we address in this paper are:
\begin{enumerate}
    \item We need to control the dynamics after the blowups in \cref{constr:main}. This is done using a cone field argument.
    \item Cylindrical contact homology is not well defined for stable Hamiltonian structures, so we must work instead with the more complicated rational symplectic field theory. Despite the absence of a grading by free homotopy classes, we can still extract enough information to make the argument go through.
\end{enumerate}
Finally, let us mention \cite{colin.honda.ReebVectorFields} which is similar in spirit to this paper and performs a computation of cylindrical contact homology in a pseudo-Anosov contact setting. In their language, this paper carries out their program to prove the Weinstein conjecture in the case of contact structures on rational homology 3-spheres with pseudo-Anosov open books having fractional Dehn twist coefficient $2/n$ or larger; see \cite[Corollary 2.4, Remark 2.5]{colin.honda.ReebVectorFields}.

\subsection{Finiteness conjecture for pseudo-Anosov flows}
One source of motivation for \cref{thm:main} is the following longstanding conjecture:
\begin{conj}\label{conj:anosov_finiteness}
    On any given closed, oriented 3-manifold, there are a finite number of transitive pseudo-Anosov flows up to orbit equivalence.
\end{conj}
Barthelm\'{e}, Mann, and Bowden solved the finiteness problem for skew-$\R$-covered Anosov flows by reducing the question to the finiteness of tight contact structures with zero Giroux torsion \cite{barthelme.mann.OrbitEquivalencesMathbb}. This finiteness problem was previously resolved by Colin, Giroux, and Honda \cite{colin.giroux.ea.FinitudeHomotopiqueIsotopique}. In order to use the same idea in our setting, we need a criterion for deciding when our stable Hamiltonian structures are of contact type.

\begin{prop}\label{prop:tight}
    Suppose $\phi$ admits a positive Birkhoff section. Then $(\omega_\phi, \lambda_\phi)$ is possibly-non-exact stable homotopic to the Reeb flow of a contact structure. Moreover, the contact structure is tight and has zero Giroux torsion.
\end{prop}

Thus, we may combine \cite{colin.giroux.ea.FinitudeHomotopiqueIsotopique} with \cref{thm:main} to obtain:
\begin{mainthm}\label{thm:contact_finiteness}
    On any closed, oriented rational homology 3-sphere, there are finitely many transitive pseudo-Anosov flows up to orbit equivalence which admit a positive Birkhoff section.
\end{mainthm}

It is now tempting to attack \cref{conj:anosov_finiteness} by proving a finiteness result for stable Hamiltonian structures. There is so far no natural candidate definition for ``tight stable Hamiltonian structure''. However, the stronger condition of hypertightness generalizes easily. We say that a stable Hamiltonian structure is \emph{hypertight} if its Reeb flow has no contractible orbits. The stable Hamiltonian structures produced by \cref{constr:main} are hypertight, so we ask:

\begin{ques}Are there finitely many hypertight stable Hamiltonian structures up to stable homotopy on any given irreducible atoroidal rational homology 3-sphere?
\end{ques}

\begin{subsection}{Dehn surgery problems}
      One of the most enduringly popular questions in low dimensional topology is that of enumerating the ways in which a given manifold can be realized as Dehn surgery on a knot in $S^3$. A reason for this popularity is the wide variety of tools that may be brought to bear on the problem, including hyperbolic geometry, Floer homology, foliation theory, and character varieties. Problem 3.6D from Kirby's problem list asks:
    \begin{ques}
        Is there a 3-manifold $M$ and a slope $n$ such that $M$ arises as $n$-surgery for along infinitely many different knots $K\in S^3$?
    \end{ques}

    Osoinach answered this question in the affirmative by constructing an infinite family of hyperbolic knots with the same 0-surgery \cite{osoinach.ManifoldsObtainedSurgery}. Abe, Jong, Luecke, and Osoinach extended this construction to other integer slopes \cite{abe.jong.ea.InfinitelyManyKnots}. However, the finiteness question remains open for other slopes. We make partial progress on the problem for many non-integer slopes. Let $S^3_{p/q}(K)$ denote the result of $p/q$ Dehn surgery on a knot $K$ in $S^3$.

    \begin{mainthm}\label{thm:surgery}
        Define the set of slopes $$V=\{p/q \mid \gcd(p,q)=1,\, q \geq 6,\text{ and } \,p \neq 0, \pm 1, \pm 2 \mmod q\}.$$
        For any 3-manifold $M$, there are at most finitely many fibered hyperbolic knots $K\subset S^3$ and slopes $p/q\in V$ with $S^3_{p/q}(K)\cong M$.
    \end{mainthm}

    If \cref{conj:anosov_finiteness} holds, then the hypothesis ``fibered'' can be removed using the pseudo-Anosov flow constructed by Gabai and Mosher on a hyperbolic knot complement \cite{mosher.LaminationsFlowsTransverse}. Note that by Hanselman's work on the cosmetic surgery conjecture, there is at most one slope $p/q$ for each knot $K$ such that $S^3_{p/q}(K)\cong M$ \cite{hanselman.HeegaardFloerHomology}. Despite the numerous successes of Floer homology in studying Dehn surgery, \Cref{thm:surgery} appears out of reach of those techniques because there exist infinite sequences of hyperbolic knots (even fibered hyperbolic knots) with the same knot Floer homology \cite{hedden.watson.GeographyBotanyKnot}.

      Say that a slope $p/q$ is \emph{characterizing} for $K$ if $S^3_{p/q}(K)$ does not arise as $p/q$ Dehn surgery for any other knot in $S^3$. Intuition from hyperbolic geometry suggests that for a hyperbolic knot $K$ and $|p|+|q|$ large, $p/q$ should be characterizing for $K$. Thurston's hyperbolic Dehn surgery theorem says that for $|p|+|q|$ sufficiently large, $S^3_{p/q}(K)$ is hyperbolic and the Dehn surgery core is the shortest geodesic in the resulting manifold. Thus, the geometry of $S^3_{p/q}(K)$ remembers the manner in which it was obtained from $S^3$. By the Gordon--Luecke theorem, we may return to $S^3$ in a unique way by performing surgery on this short geodesic. The problem with this argument is that the required lower bound on $|p|+|q|$ is not universal, so there might be another sneaky $K'$ with $S^3_{p/q}(K')\cong S^3_{p/q}(K)$ and whose corresponding Dehn surgery core is not the shortest geodesic. Lackenby explains how to circumvent these problems in \cite{lackenby.EveryKnotHas}, and proves that every knot in $S^3$ has some characterizing slope.

    Our approach to \cref{thm:surgery} replaces hyperbolic geometry with pseudo-Anosov geometry. Taking the place of Mostow rigidity is \cref{thm:contact_finiteness}. The role of short geodesics is played by the singular orbits. Thurston's hyperbolic Dehn surgery theorem is replaced with Fried--Goodman surgery. The exceptional slopes are those that intersect the degeneracy slope at most once. Our good control over the degeneracy slope gives us the universal bounds that were missing in the hyperbolic setting.
\end{subsection}

\end{section}

\section*{Acknowledgements}
I would like to thank the organizers and participants of the ``Symplectic Geometry and Anosov Flows'' workshop in Heidelberg for a lively week, many interesting conversations, and helpful feedback on drafts of this paper. I benefited a lot from experts on both the dynamics side and the symplectic side.

\begin{section}{Preliminaries}
    $M$ is always a closed, oriented, irreducible 3-manifold. We use $\phi$ to denote a transitive pseudo-Anosov flow on $M$. All pseudo-Anosov flows in the paper are assumed to be transitive. Note that transitive pseudo-Anosov flows do not exist on reducible 3-manifolds. An orbit of $\phi$ is called \emph{nonrotating} if the return map along the orbit returns each sector bounded by stable prongs to itself; otherwise it is called \emph{rotating}.
    \subsection{Stable Hamiltonian structures} A stable Hamiltonian structure (or SHS for short) on a closed, oriented 3-manifold $M$ is a pair $(\omega, \lambda)$ where $\omega$ is a 2-form, $\lambda$ is 1-form, and
\begin{align*}
    d\omega &= 0\\
    \omega \wedge \lambda &> 0\\
    \ker(d\lambda) &\subset \ker(\omega)
\end{align*}
We say that $\lambda$ \emph{stabilizes} $\omega$. The \emph{Reeb vector field} of $(\omega,\lambda)$ is the unique vector field $R$ in $\ker(\omega)$ with $\lambda(R)=1$. The Reeb flow preserves the volume form $\omega \wedge \lambda$.  A useful mnemonic to remember the conditions is that they look Poincar\'e dual to the data of a volume preserving flow with a Birkhoff section:
\begin{align*}
    \omega &\leadsto \text{nowhere vanishing flow $\phi$}\\
    d\omega = 0 &\leadsto \text{$\phi$ is volume preserving}\\
    \lambda &\leadsto \text{pages of an open book}\\
    \omega \wedge \lambda > 0 &\leadsto \text{$\phi$ is transverse to the pages of the open book}\\
    \ker(d\lambda) \subset \ker(\omega) &\leadsto \text{boundary of pages are flowlines}
\end{align*}
 A contact form $\alpha$ is an SHS with $\omega=d\alpha$ and $\lambda=\alpha$. A fibered 3-manifold supports an SHS for which $\ker(\lambda)$ is a fibration, $d\lambda=0$, and $\omega$ a closed 2-form positive on the leaves of the fibration. A manifold with an SHS breaks into regions where $\lambda \wedge d\lambda > 0$ (the positive contact region), $\lambda \wedge d\lambda < 0$ (the negative contact region), and $d\lambda=0$ (the integrable region).

 The \emph{admissible interval} of an SHS is the largest interval $(a,b)\subset \R$ such that $\omega+t\,d\lambda$ is a nowhere vanishing 2-form for all $t\in (a,b)$. From the persective of Reeb dynamics, the different SHS coming from different choices of $t\in(a,b)$ behave the same. A \emph{strong symplectic cobordism} between two SHS $(\omega_1,\lambda_1)$ and $(\omega_2,\lambda_2)$ is a symplectic manifold $([0,1]\times M,\Omega)$ such that $\Omega|_{0\times M} = \omega_1 + \varepsilon_1 \lambda_1$ and $\Omega|_{1\times M} = \omega_2 + \varepsilon_2 \lambda_2$, and $\varepsilon_1$ and $\varepsilon_2$ are in the admissible intervals of $(\omega_1,\lambda_1)$ and $(\omega,\lambda_2)$ respectively. A \emph{trivial symplectic cobordism} from $(\omega,\lambda)$ to $(\omega,\lambda)$ is a symplectic manifold $([0,\varepsilon]\times M,\Omega)$ with $\Omega = \omega + t\,d\lambda$ for $\varepsilon$ a small enough constant to make it a symplectic cobordism. All the symplectic cobordisms in this paper are topologically trivial.

We say that two stable Hamiltonian structures $(\omega_0, \lambda_0)$ and $(\omega_1, \lambda_1)$ are \emph{possibly-non-exact stable homotopic} if they are homotopic through a smooth 1-parameter family of SHS $(\omega_t,\lambda_t)_{t\in [0,1]}$. We say that they are \emph{exact stable homotopic} if $\frac d {dt} \omega_t$ is exact for all $t\in [0,1]$. We say that $(\omega_0, \lambda_0)$ and $(\omega_1,\lambda_1)$ are \emph{cobordism equivalent} if there are strong symplectic cobordisms $X$ and $Y$ in both directions such that $X \circ Y$ and $Y \circ X$ are homotopic to trivial cobordisms. Unfortunately, this is not an equivalence relation, because strong cobordisms cannot always be composed in the stable Hamiltonian setting! To fix this, we say that $(\omega_0, \lambda_0)$ and $(\omega_1,\lambda_1)$ are \emph{broken cobordism equivalent} if there is a sequence of cobordism equivalences starting at $(\omega_0,\lambda_0)$, ending at $(\omega_1,\lambda_1)$, and passing through only Morse-Bott SHS. Cieliebak and Volkov showed that if two SHS are exact stable homotopic, then they are broken cobordism equivalent \cite[Corollary 7.27]{cieliebak.volkov.FirstStepsStable}. Many other foundational facts about stable Hamlitonian structures were proven by Cieliebak and Volkov in \cite{cieliebak.volkov.FirstStepsStable}. See their Section 7 for a more thorough treatment of homotopies and cobordisms.

\subsection{Birkhoff sections}
A Birkhoff section for a flow on $M$ is a compact oriented surface embedded in $M$ such that its boundary is a collection of closed orbits (oriented either positively or negatively), its interior is positively transverse to the flow, and it intersects every orbit in forward and backwards time. We say that a Birkhoff section is \emph{positive} if all of its boundary components are positively oriented flowlines.
We will often use the following fact:
\begin{prop}\label{prop:avoid}
    Any transitive pseudo-Anosov flow on a closed, oriented 3-manifold admits a Birkhoff section. Moreover, we can choose the Birkhoff section so that any given finite collection of orbits does not intersect the boundary of the Birkhoff section.
\end{prop}
See \cite{fried.TransitiveAnosovFlows} for the original argument, and \cite{brunella.SurfacesSectionExpansive} for the generalization to pseudo-Anosov flows. The fact that we can avoid a given finite collection of orbits is not explicitly stated, but it is clear from the strategy of proof. One builds a Birkhoff section by piecing together ``partial sections", and they construct a partial section whose interior contains any given point in $M$ and whose projection to the orbit space is as small as desired, and therefore can be taken to miss any finite collection of orbits.

\subsection{Compatible open books}
A \emph{signed open book decomposition} of a 3-manifold $M$ is an oriented link $L$ along with a fibration $\pi:M\setminus L \to S^1$, such that the $\pi$-images of all meridians are non-contractible in $S^1$. Note that the orientations on the components of $L$ may differ from the orientation inherited as the boundary of the pages. A Birkhoff section gives rise to an open book decomposition of the ambient 3-manifold where the oriented link is the binding oriented with the flow. Given a signed open book decomposition $\mathcal B$, we say that an SHS $(\omega, \lambda)$ is $\emph{compatible}$ with $\mathcal B$ if the pages of $\mathcal B$ are Birkhoff sections for the Reeb flow of $(\omega, \lambda)$ and the binding components are oriented closed orbits of the Reeb flow.

\begin{prop}[{\cite[Theorem 4.2]{cieliebak.volkov.FirstStepsStable}}]\label{prop:openbook}
    Suppose $(\omega_1, \lambda_1)$ and $(\omega_2,\lambda_2)$ are compatible with the same signed open book decomposition. Then $(\omega_1,\lambda_1)$ and $(\omega_2,\lambda_2)$ are possibly-non-exact stable homotopic. If in addition $[\omega_1]=[\omega_2]$, then the two stable Hamiltonian structures are exact stable homotopic.
\end{prop}

\end{section}

\begin{section}{Constructions}
\subsection{Blowup of hyperbolic fixed points}\label{subsec:blowup}
Our first task is to construct a smooth model for the blowup of a hyperbolic fixed point (or more generally, a $k$-pronged pseudo-hyperbolic fixed point). Equip $\R^2$ with its standard symplectic form. Let $g(x,y)$ be a bump function supported on the unit disk with $\Z/2$ rotational symmetry around the origin. Fix parameters $A$ and $\varepsilon>0$. Define the Hamiltonian $H_{A} = x^2-y^2 + A\varepsilon^2\,g(x/\varepsilon,y/\varepsilon)$. Let $f_{A,t}:\R^2 \to \R^2$ be the time $t$ map of the associated Hamiltonian flow.

$f_{0,t}$ has a single hyperbolic fixed point at the origin whose Lyapunov exponents are $\pm t$. When $|A|$ is sufficiently large, $f_{A,t}$ has two hyperbolic fixed points $x_h$ and $x_h'$ and an elliptic fixed point $x_e$. The elliptic fixed point sits inside a region bounded by the separatrices connecting $x_h$ and $x_h'$; we call this region the \emph{eye} because of its shape. The parameter $\varepsilon$ controls the size of the eye. The sign of $A$ dictates whether the flow rotates clockwise or counterclockwise inside the eye. In this situation, we call $f_{A,t}$ a \emph{clockwise or counterclockwise blow-up} of $f_{0,t}$.

This construction was previously used by Cotton-Clay to smoothen pseudo-Anosov homeomorphisms \cite{cotton-clay.SymplecticFloerHomology}. He was interested only in order 1 fixed points, but we need more: we want to know that all of the dynamics outside the eye is unchanged by blowup. The next lemma will be useful in a cone field argument to show that the dynamics outside the eye remains hyperbolic. We say that a diffeomorphism $h$ \emph{strictly contracts} a cone field $C$ if at every point $p$ we have $$Dh(\overline{C(p)}) \subset \interior \, C(h(p)).$$

\begin{figure}
    \centering
    \includegraphics[width=4.5in]{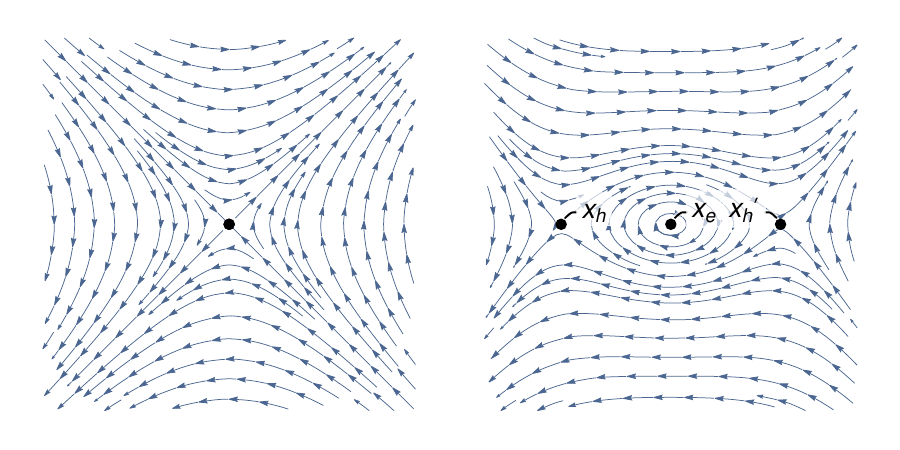}
    \caption{A clockwise blowup}
\end{figure}
\begin{lem}\label{lem:cone}
    Fix a value of $A$. Let $B_r$ denote the radius $r$ disk around the origin. Let $f_{A}^{ret,r}$ be the first return map of the flow $f_{A,t}$ to $\R^2 \setminus B_r$. For every $R>0$, there exists $\varepsilon$ sufficiently small that $f_{A}^{ret,R}$ strictly contracts the cone field $$C=\{(x,y,v_x,v_y)\in T\R^2 \mid v_xv_y > 0\text{ and } x^2+y^2=r^2\}.$$
\end{lem}
\begin{proof}
At the outset, let us take $\varepsilon < 0.1R$. The flow of $f_{A,t}$ agrees with the flow of $f_{0,t}$ on the annulus $10\varepsilon \leq \sqrt{x^2+y^2} \leq R$. The flow of $f_{0,t}$ is hyperbolic and contracts the desired cone field, so we need only be concerned with flowlines which enter $B_{10\varepsilon}$. Note that a flowline of $f_{A,t}$ enters and leaves $B_{10\varepsilon}$ at most once. So let us consider a flowline which enters $B_{R}$ at time $t_0$ at point $p_0$, enters $B_{10\varepsilon}$ at time $t_1$ at point $p_1$, exits $B_{10\varepsilon}$ at time $t_2$ at point $p_2$, and exits $B_{R}$ at time $t_3$ at point $p_3$. We can express $Df_A^{ret, R}(y_0)$ as the product of the derivatives of the flow along these segments.

\begin{align}
    Df_A^{ret, R}(p_0) &= Df_{A,t_3-t_2}(p_2) Df_{A,t_2-t_1}(p_1) Df_{A,t_1-t_0}(p_0)\\
    &=\exp{\left( (t_3-t_2) \begin{pmatrix} 0 & 1\\ 1 & 0 \end{pmatrix} \right)} Df_{A}^{ret, 10\varepsilon}(p_1) \exp{\left( (t_1-t_0) \begin{pmatrix} 0 & 1\\ 1 & 0 \end{pmatrix} \right)} \label{eq:threelegs}
\end{align}

$H_A$ is designed to have scale invariance with respect to $\varepsilon$: doubling $\varepsilon$ is the same as scaling $x$ and $y$ by a factor of 2 and multiplying $H_A$ by a constant factor. Multiplying $H_A$ by a constant factor does not change the flow, so the norm of $Df_{A}^{ret,10\varepsilon}$ acting on tangent vectors is independent of $\varepsilon$. Therefore, the middle term of \cref{eq:threelegs} is bounded above in norm independent of $\varepsilon$.

As $\varepsilon\to 0$, the minimum possible values of $t_1-t_0$ and $t_3-t_2$ go to infinity; it takes a long time for flowlines to reach a $10\varepsilon$ ball around the origin. Therefore, the first and third terms in \cref{eq:threelegs} dominate in the limit $\varepsilon\to 0$ and the product is a hyperbolic matrix which contracts the cone field $C$.
\end{proof}
\begin{lem}\label{lem:cone2}
    Fix a value of $A$ and $R$. There is $\varepsilon$ small enough that the cone field $$C=\{(x,y,v_x,v_y)\in T\R^2 \mid v_xv_y > 0\}$$ extends to a cone field on the complement of the eye which is strictly contracted by the flow of $f_{A,t}$.
\end{lem}
\begin{proof}
Extend $C$ into $B_r$ by pushing it forward into $B_R$ by the flow $f_{A,t}$. By \cref{lem:cone2} we get a cone field which is contracted by the flow of $f_{A,t}$. The cone field suffers from two problems, first that the contraction is not strict, and second that it is not continuous where the flow exits $B_R$. These problems can be simultaneously remedied by a slow expansion of the cones along flowlines of $B_R$.
\end{proof}

It is possible to choose $g$ so that $H_1=a-b(x^2 + y^2)$ in a neighbourhood of the origin for some constants $a,b$, so that $f_1$ is exactly a counterclockwise rotation on a neighbourhood of the origin. Similarly, we could arrange that $f_{-1}$ is a clockwise rotation in a neighbourhood of the origin. With this choice, the lifts of $f_1$ or $f_{-1}$ to branched covers around the origin are still smooth. Because the entire construction has $\Z/2$ rotational symmetry around the origin, we may even take $n+\frac 1 2$-fold covers around the origin. Finally, we can post-compose with a rotation around the origin which preserves the Hamiltonian. This completes our construction of a smooth, area preserving blowup of a $k$-pronged pseudo-hyperbolic fixed point.
 
The entire construction above may also be interpreted as a blowup operation at a closed orbit of a flow. Define $\overline{f}_{A,t}: \R^2 \times S^1 \to \R^2 \times S^1$, by $\overline{f}_{A,t}(x,y,\theta) = (f_{A,t}(x,y),\theta+t)$. Then for each value of $A$, the other variable $t$ parametrizes a flow. We say that $A$ parametrizes an interpolation between the flow near a standard hyperbolic closed orbit and a blowup of the flow. We continue to refer to the suspension of the eye as the eye.

There is a second kind of blowup that can performed at a nonrotating $k$-prong singular fixed point. Such a singularity can be generically perturbed through area preserving maps to $k-2$ ordinary hyperbolic fixed points connected by a tree of saddle connections; we call this a \emph{nonrotating blowup}. See \cref{fig:blowup2}. To ensure that this map is area preserving, the perturbation can be done by modifying the corresponding monkey saddle Hamiltonian to one with nondegenerate critical points. See \cite[Lemma 3.4, Lemma 3.5]{cotton-clay.SymplecticFloerHomology} for a more detailed construction. This perturbation breaks rotational symmetry so, as in \cite{cotton-clay.SymplecticFloerHomology}, we may use it only at nonrotating singularities.

\begin{figure}
    \centering
    \includegraphics[width=4.5in]{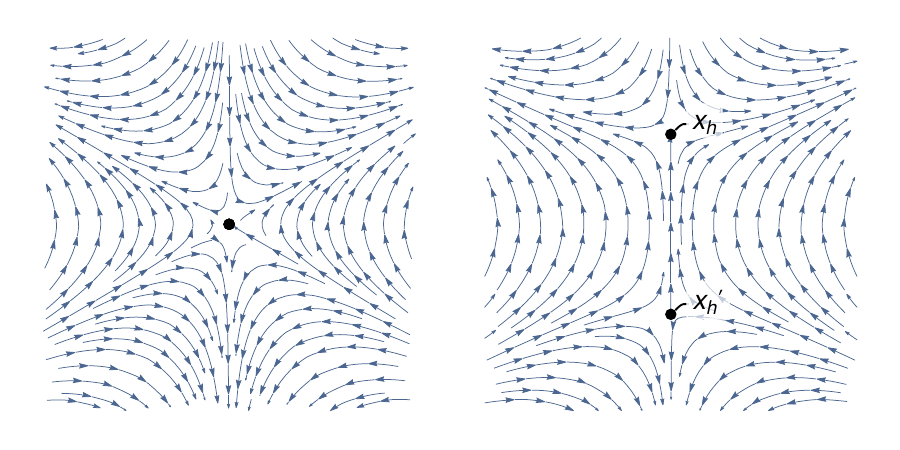}
    \caption{A nonrotating blowup}\label{fig:blowup2}
\end{figure}

An argument similar to the one in \cref{lem:cone} and \cref{lem:cone2} proves:

\begin{lem}\label{lem:cone3}Let $k\geq 3$. Let $\widetilde{f_{t}}:\R^2 \to \R^2$ be the standard model for a $k$-prong singular fixed point, obtained by taking a $(k-2)/2$-fold branched cover of $f_{0,t}$. Let $\widetilde C$ be the lift of the cone field $C$ from \cref{lem:cone}. Then for any $R>0$, there is a nonrotating blowup $\widetilde{f_t}'$ supported on $B_R$ such that there is a cone field $C'$ on the complement of the tree of saddle connections which agrees with $\widetilde C$ outside $B_R$ and is strictly contracted by the flow of $\widetilde{f_t}'$.
\end{lem}

\subsection{Transverse Dehn surgery}
Let $S^1\times D^2$ be a solid torus with coordinates $\psi$ on $S^1$ and $r,\theta$ on $D^2$. When discussing slopes on an $r=const$ torus, we say that the $\psi$ direction has slope $0$ and then $\theta$ direction has slope $1/0$.
The next lemma should be familiar from the standard picture of the Reeb flow near the binding of an open book decomposition.
\begin{figure}
    \centering
    \includegraphics[width=2.75in]{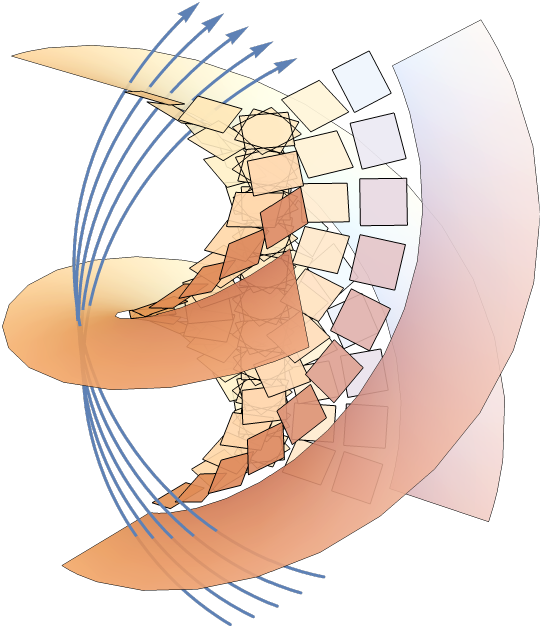}
    \caption{$\ker{\lambda}$ is a positive or negative contact structure near the binding of $\mathcal B$ and a foliation on the interior of the pages.}\label{fig:confoliation}
\end{figure}
\begin{lem}\label{lem:solidtorus}
For any choice of $p,q > 0$ and a smooth function $s:[0,1] \to \R_{\geq 0}$ which extends to a smooth even function around 0, there is a stable Hamiltonian structure $(\omega, \lambda)$ on $S^1 \times D^2$ satisfying
\begin{enumerate}
    \item its Reeb flow preserves concentric tori, is linear on each torus, and has slope $s(r)$ at radius $r$.
    \item $d\lambda=0$ for $r\in [1/2,1]$
    \item $\lambda = q d\theta + p d\psi$ for $r\in [1/2,1]$
    \item $\lambda \wedge d\lambda \geq 0$.
\end{enumerate}
\end{lem}

\begin{proof}
Fix $f(r)$ satisfying
\begin{enumerate}
    \item $f' \geq 0$
    \item $f(r)=r^2$ near $r=0$ 
    \item $f(r)=q$ for $r\in [1/2,1]$
\end{enumerate}

Define $g(r)$ by $g(r) = p+\int_r^1 f'(x)s(x)dx$. Then $g$ satisfies
\begin{enumerate}
    \item $g'(r)=-f'(r)s(r) \leq 0$
    \item $g(r)=p$ for $r\in [1/2,1]$
    \item $g(r) > 0$
    \item $g(r)$ extends to an even smooth function in a neighbourhood of $r=0$
\end{enumerate}
    Consider the 1-form $\lambda=f(r) d\theta + g(r)d\psi$. This 1-form satisfies the second and third conditions in the statement of the lemma. By construction, $f$ and $g$ extend to a smooth even functions near $r=0$. Therefore, $\lambda$ is smooth even at $r=0$. Since $f$ is non-negative and $g$ is positive, $\lambda$ evaluates positively on positive slopes. Therefore, there is a vector field $R$ which is linear on concentric tori, has slope $s(r)$ at radius $r$, and is normalized with $\lambda(R)=1$. Since $f$ is non-decreasing and $g$ is non-increasing, $\ker \lambda$ is a confoliation. We have $d\lambda = f'(r)dr \wedge d\theta + g'(r) dr \wedge d\psi$. Since $s(r) = -g'(r)/f'(r)$, we have $R\subset \ker(d\lambda)$. Let $V$ be any cylindrically symmetric volume form. Now let $\omega = \iota_R V$. Then $(\omega, \lambda)$ is an SHS with Reeb vector field $R$.
\end{proof}

\begin{constr}[Transverse Dehn surgery] \label{constr:surgery}
Let $M$ be a 3-manifold with a stable Hamiltonian structure having a region stable Hamiltonian equivalent to $S^1\times D^2$ with $\lambda = d\psi$ and $\omega$ dual to the suspension of a flow which preserves concentric circles and rotates clockwise. Then for $p,q > 0$, we can perform a $p/q$ Dehn surgery producing a new SHS. Drill out a solid torus and glue in the one constructed in \cref{lem:solidtorus}. That lemma allows the flexibility to match both the Reeb flow and $\lambda$ on the overlapping $T^2 \times [0,1]$. We can match $\omega$ on the overlap because scaling $\omega$ by any function which depends only on $r$ results in a new SHS. Similarly, if the suspension flow rotates counterclockwise, we can do a $-p/q$ Dehn surgery using the mirror image of the solid torus form \cref{lem:solidtorus}.
    
\end{constr}

\subsection{The main construction} Now let us explain \cref{constr:main}, how to go from a pseudo-Anosov flow to a stable Hamiltonian structure. Let $\phi$ be a transitive pseudo-Anosov flow. Choose a Birkhoff section $S^\circ$ for $\phi$ whose boundary orbits are non-singular. Let $\gamma_1,\dots,\gamma_n$ be the boundary components of $S^\circ$. Each $\gamma_i$ has a sign $s_i=\pm 1$ indicating whether the orientation induced as a boundary component of $S^\circ$ agrees or disagrees with the orientation of the flow.

Let $S$ be the closed surface obtained by collapsing each boundary component $\gamma_i \subset \partial S^\circ$ to a point which we call $x_i$. Let $\varphi^\circ:\interior(S^\circ) \to \interior( S^\circ)$ be the first return map to $S^\circ$ and let $\varphi:S\to S$ be the induced pseudo-Anosov homeomorphism of $S$. Now we would like to perform a clockwise or counterclockwise blowup operation on all the singular points of $\varphi$ as well as $x_1,\dots,x_n$. When $s_i=1$, we will perform a clockwise blowup at $x_i$, and when $s_i=-1$, we will perform a counterclockwise blowup. At the nonrotating singularities, perform a nonrotating blowup. At rotated singularities, perform either a clockwise or counterclockwise blowup. Call the blownup map $ \varphi_{blowup}$, and let $\varphi_{blowup}^\circ$ be the restriction of $\varphi_{blowup}$ to the complement of the eyes and the saddle connection trees.

\begin{prop}
    There is a choice of blowup parameters so that $\varphi_{blowup}$ is semiconjugate to $\varphi$ by the map which blows down all the eyes and saddle connection trees.
\end{prop}
\begin{proof}
    Use \cref{lem:cone2} and \cref{lem:cone3} to select blowups with very small support. With this choice, there is a cone field $C$ on the complement of all the eyes and saddle connection trees which is strictly contracted by $\varphi_{blowup}$. It follows that $\varphi_{blowup}^\circ$ is a hyperbolic diffeomorphism in the same mapping class as $\varphi^\circ$. There is at most one hyperbolic diffeomorphism up to topological conjugacy in each mapping class, so $\varphi_{blowup}^\circ$ and $\varphi^\circ$ are topologically conjugate. The proposition follows.
\end{proof}

Since $\varphi_{blowup}$ is a smooth area preserving diffeomorphism, the mapping torus of $\varphi_{blowup}$ naturally carries a stable Hamiltonian structure whose Reeb flow has first return map equal to $\varphi_{blowup}$. Our choice of the signs $s_i$ ensures that we can do the transverse Dehn surgery from \cref{constr:surgery} to return from the mapping torus of $\varphi_{blowup}$ to $M$. We usually denote the resulting stable Hamiltonian structure $(\omega_\phi, \lambda_\phi)$, or $(\omega_{\phi,\mathcal B}, \lambda_{\phi,\mathcal B})$ when we want to specify which open book decomposition $\mathcal B$ was used in the construction.

\begin{rmk}[Smooth invariant volume forms]
        Given a transitive Anosov flow, Asaoka solved the tricky problem of finding an orbit equivalent flow which preserves a smooth volume form \cite{asaoka.InvariantVolumesCodimensionone}. This technique is key in Marty's proof that skew-$\R$-covered flows are orbit equivalent to Reeb flows \cite{marty.SkewedAnosovFlows}. Our blowups, though aesthetically unappealing, help us to skirt this issue. We start with a suspension flow on the mapping torus of $\varphi_{blowup}$, where the smooth volume form is unique and explicit. Then, instead of doing Fried surgery which does not preserve smoothness of the volume form, we do a Dehn surgery in the eye which is a smooth operation. The cost is that the smooth structure on the manifold depends on the choice of initial Birkhoff section. But smooth structures on 3-manifolds are unique up to homeomorphism, and Birkhoff sections are always smoothable, so we can transfer the flow and Birkhoff section back to the original smooth manifold.
\end{rmk}

\subsection{Some stable homotopies}
Let us catalogue some different stable homotopies that will be useful later.
\begin{enumerate}[label=Type \arabic*]
    \item\label{item:type1} If $\lambda_1$ and $\lambda_2$ are any two stabilizing 1-forms for $\omega$, then so is any linear combination of $\lambda_1$ and $\lambda_2$. Therefore, there is an exact stable homotopy between $(\omega, \lambda_1)$ and $(\omega,\lambda_2)$.
    \item Suppose $(\omega_1, \lambda)$ and $(\omega_2, \lambda)$ are two SHS with the same $\lambda$. Suppose $\omega_1=\omega_2$ except on the integrable region (that is, the region where $d\lambda =0$). Then every linear combination of $\omega_1$ and $\omega_2$ is stabilized by $\lambda$, so $(\omega_1,\lambda)$ and $(\omega_2, \lambda)$ are stable homotopic by a linear interpolation.
    \item \label{item:type3} Suppose we have a solid torus of radius $R$ with a cylindrically symmetric SHS. Suppose we have a non-negative function on the solid torus depending only on radius, $h(r):[0,R]\to \R_{\geq 0}$, satisfying $h(r)=1$ in a neighbourhood of $r=R$. Then we can scale $\omega$ by $h(r)$. The homology class $[\omega^*]$ changes by a multiple of $[\gamma]$ where $\gamma$ is the closed orbit at the centre of the tube.
    \item\label{item:type4} Suppose we have a solid torus with an SHS satisfying $\lambda =d\psi$, where $\psi$ parameterizes the $S^1$ direction. Then the Reeb flow moves at constant speed in the $\psi$ direction, and so is determined by the first return map to a $D^2$ slice. The first return map together with an invariant area form on $D^2$ determine the SHS. As in type 2, any two such SHS agreeing in a neighbourhood of the boundary of the solid torus are stable homotopic. If the invariant area forms on a $D^2$ slice are the same, then the stable homotopy is exact.
\end{enumerate}

\begin{subsection}{Exact vs possibly-non-exact stable homotopy}\label{sec:exact-non-exact}
    $(\omega_\phi, \lambda_\phi)$ is not unique up to exact stable homotopy because exact stable homotopy does not change the cohomology class of $\omega_\phi$, and this cohomology class is not uniquely specified by $\phi$. However, after fixing a choice of $[\omega_\phi]$, the stable Hamiltonian structure $(\omega_\phi, \lambda_\phi)$ is unique up to exact stable homotopy. The simplest example of this phenomenon is the case of a suspension pseudo-Anosov flow. There are in general many fibrations transverse to $\phi$ catalogued by points of the corresponding fibered face of the Thurston norm ball. Each fibration gives rise to a different stable Hamiltonian structure where $\omega$ is a smooth invariant measure for the monodromy and $\ker \lambda$ is the fibration. Let us state a more precise version of \cref{constr:main} which takes into account this choice. Given a pseudo-Anosov flow $\phi$, let $C_\phi\subseteq H_1(M,\R)$ be the closed cone generated by periodic orbits of $\phi$. We call $C_\phi$ the \emph{cone of homology directions} of $\phi$. Any $\phi$-invariant measure gives a class in $H_1(M,\R)$. Any such class is contained in $C_\phi$, and moreover, periodic orbits are dense in the projectivization of $C_\phi$ \cite{plante.HomologyClosedOrbits}. We denote the Poincar\'e dual of $\omega$ by $\omega^*$. 

    \begin{constr}\label{constr:main2}
        Let $x$ be a point on the interior of the cone of homology directions of $\phi$ in $H_1(M,\R)$. Then one can blow up finitely many orbits of $\phi$ to obtain the Reeb flow of a stable Hamiltonian structure $(\omega_\phi, \lambda_\phi)$ with $\omega_\phi^*=x$. Different choices of blowup give rise to stable Hamiltonian structures which are exact stable homotopic.
    \end{constr}
    
    \begin{rmk}\label{rmk:nonexact}
    The main reason for restricting to exact stable homotopies in \cref{thm:main} is to get invariance of symplectic field theory. The usual continuation map argument requires a topologically trivial symplectic cobordism, which can only exist when the cohomology class of $\omega$ does not change. However, one can sometimes use a bifurcation argument instead of a continuation map argument to prove invariance. For example, in the case of atoroidal fibered 3-manifolds, Cotton--Clay explains that cylindrical SFT is invariant even under non-exact homotopies which preserve the fibration \cite[Section 2]{cotton-clay.SymplecticFloerHomology}. He uses Lee's bifurcation analysis which classifies how a closed orbit can bifurcate during a 1-parameter family of possibly non-exact deformations. The atoroidal condition prohibits holomorphic cylinders from a closed orbit to itself. Thus, there is some hope to remove the exactness condition in \cref{thm:main}, at least in the atoroidal setting.
    \end{rmk}
\end{subsection}

First we prove the existence part of \cref{constr:main2}.

\begin{lem}
    For any $x\in int \, C_\phi$, there is a choice of blowup parameters so that $\omega_\phi^* = x$.
\end{lem}

\begin{proof}
Start with any $(\omega_\phi, \lambda_\phi)$ as constructed above. Find a collection of nonsingular closed orbits $\gamma_1, \dots, \gamma_k$ such that $x$ is on the interior of the convex cone generated by $[\gamma_1], \dots, [\gamma_k]$; we can do this because closed orbits are dense in the projectivization of $C_\phi$. Perform a clockwise or counterclockwise blowup on these orbits if they are not already blown up. Now we will arrange that most of the mass of $\omega \wedge \lambda$ lives inside the eyes of $\gamma_1,\dots,\gamma_n$. For small enough choice of the blowup parameter $\varepsilon$, $[\omega_\phi^*]$ is still on the interior of the cone generated by $[\gamma_1],\dots, [\gamma_k]$. Now scale down $\omega_\phi$ by a very small constant factor so that $x-\omega_\phi^*$ is on the interior of the convex cone generated by the $[\gamma_i]'s$. Finally, we can correct for the difference between $x$ and $\omega_\phi^*$ by adding a multiple of $[\gamma_i]$ to $[\omega_\phi^*]$ for each $i$. This can be achieved by scaling $\omega_\phi$ inside the corresponding eye as in a \ref{item:type3} stable homotopy.
\end{proof}

We finish by proving the uniqueness statements in \cref{constr:main} and \cref{constr:main2}.

\begin{lem}
Let $\mathcal B$ and $\mathcal B'$ be any two signed open book decompositions of $M$ obtained from Birkhoff sections of $\phi$, and let $(\omega_{\phi,\mathcal B}, \lambda_{\phi,\mathcal B})$ and let  $(\omega_{\phi, \mathcal B'}, \lambda_{\phi, \mathcal B'})$ be stable Hamiltonian structures built using \cref{constr:main}. Then $(\omega_{\phi,\mathcal B}, \lambda_{\phi,\mathcal B})$ can be made compatible with $\mathcal B'$ by a possibly-non-exact stable homotopy. Moreover, $(\omega_{\phi, \mathcal B}, \lambda_{\phi,\mathcal B})$ is possibly-non-exact stable homotopic to $(\omega_{\phi, \mathcal B'}, \lambda_{\phi, \mathcal B'})$. If in addition $[\omega_{\phi, \mathcal B}]=[\omega_{\phi,\mathcal B'}]$, then the stable homotopies can be made exact.
\end{lem}
\begin{proof}
    This is obvious when the binding components of $\mathcal B'$ are disjoint from the binding components of $\mathcal B$; in this case, the blowups used in the construction of $(\omega_{\phi,\mathcal B}, \lambda_{\phi,\mathcal B})$ occur on the interior of the pages of $\mathcal B'$, and $\mathcal B'$ remains compatible with the Reeb flow of $(\omega_{\phi,\mathcal B}, \lambda_{\phi,\mathcal B})$. By \cref{prop:openbook}, $(\omega_{\phi, \mathcal B}, \lambda_{\phi,\mathcal B})$ is stable homotopic to $(\omega_{\phi, \mathcal B'}, \lambda_{\phi, \mathcal B'})$ and the only obstruction to an exact stable homotopy is the cohomolgy class of $\omega$.
    
    If $\mathcal B$ and $\mathcal B'$ do share binding components, then it is possible that the orbits in the eyes of $(\omega_{\phi,\mathcal B}, \lambda_{\phi,\mathcal B})$ are not positively transverse to the pages of $\mathcal B'$. With some work, this can be fixed by an exact homotopy in the eye similar to \ref{item:type4}, but a lazier way around this problem is to simply pass through a third Birkhoff section $\mathcal B''$ which shares no binding components with $\mathcal B$ or $\mathcal B'$. Such a Birkhoff section exists by \cref{prop:avoid}.
\end{proof}

\end{section}

\begin{section}{Proof of \cref{thm:main}}
\begin{subsection}{Closed orbits of $(\omega_\phi, \lambda_\phi)$}
\begin{thm}[{\cite[Theorem 1.3]{barthelme.fenley.ea.AnosovFlowsSame},\cite[Theorem 1.1, Proposition 1.1]{barthelme.frankel.ea.OrbitEquivalencesPseudoAnosov}}]\label{thm:reconstruction} Let $\phi$ be a transitive pseudo-Anosov flow on a hyperbolic 3-manifold. Up to orbit equivalence isotopic to the identity, $\phi$ is determined by the set of conjugacy classes  $[g] \subset \pi_1(M)$ such that either $g$ or $g^{-1}$ is represented by a closed orbit of $\phi$. If $M$ is toroidal, then $\phi$ is determined up to a finite ambiguity.
\end{thm}

In their theorem, it is really only necessary to know which primitive conjugacy classes are represented thanks to the following lemma:
\begin{lem}\label{lem:primitive}
    Let $g$ be any element of $\pi_1(M)$. Suppose $k\in \Z$. If $g^k$ is represented by a closed orbit of $\phi$, then so is either $g$ or $g^{-1}$. (Note that we are not ruling out the possibility that the orbit representing $g^k$ is a primitive orbit.)
\end{lem}
\begin{proof}
    There are several ways to prove this, with varying levels of difficulty depending on how much one knows about the structure of the orbit space of pseudo-Anosov flows. The simplest (communicated to us by Barthelm\'e and Mann) relies on Brouwer's translation theorem on the plane. The theorem states if $g$ acts on $\R^2$ without fixed points, then so does $g^k$ for any $k\neq 0$. Applying this theorem to the action of $g$ on the orbit space of $\phi$ gives the desired result.

    One can also prove the lemma by looking at the action of $g$ on other related spaces. Let $L^s$ be the leaf space of the stable foliation of the orbit space of $\phi$, and let $\overline {L^s}$ be its Hausdorffification. We consider all the leaves abutting a singular orbit to be the same leaf of $L^s$. The action of $g$ on $L^s$ has a fixed point if and only if either $g$ or $g^{-1}$ is represented by an orbit. Suppose that $g$ acts on $\overline{L^s}$ without fixed points. 
    The Hausdorffification $\overline{L^s}$ is a topological tree, meaning that it is simply connected and between every pair of points there is a unique minimal path. If $g$ acts without fixed points on such a space, then it acts as a translation along an axis. Thus $g^k$ would also act without fixed points on $\overline{L^s}$, which is a contradiction. Modulo the difference between $L^s$ and $\overline{L^s}$, this proves the desired claim.

    The proof suggested above can be simplified using the machinery of lozenges to produce an honest tree. This argument is explained in \cite{fenley.HomotopicIndivisibilityClosed}, and we briefly recall the strategy. Consider the set of all points in the orbit space fixed by $g^k$. Each point in this set corresponds with a closed orbit representing $g^{\pm k}$. This set is invariant under the action of $g$. Fenley proves that this set is naturally the vertex set of a (possibly infinite) planar tree $T$, whose edges correspond with so-called lozenges \cite[Theorem 3.5]{fenley.QuasigeodesicAnosovFlows}. Furthermore, $T$ comes with an embedding into $L^s$. We have three cases:
    \begin{enumerate}[label=Case \arabic*:]
        \item $g$ acts on $T$ as translation along an axis. This can't happen because $g^k$ would not fix a vertex.
        \item $g$ fixes a vertex of $T$. This vertex is an orbit representing $g$ or $g^{-1}$.
        \item $g$ acts on $T$ as a rotation around the midpoint of some edge. This edge may be interpreted as interval in $L^s$, so $g$ fixes a point in $L^s$ as desired.
    \end{enumerate}
\end{proof}

The next lemma is another consequence of the theory of lozenges, see \cite[Theorem 3.3]{fenley.QuasigeodesicAnosovFlows} for the original argument and \cite[Propostion 2.24]{barthelme.frankel.ea.OrbitEquivalencesPseudoAnosov} for the statement in the case of pseudo-Anosov flows.
\begin{lem}[Fenley]\label{lem:lozenge}
    Suppose $\gamma_1$ and $\gamma_2$ are distinct closed orbits of $\phi$ in the same free homotopy class. Then $\gamma_1$ and $\gamma_2$ are each either a (possibly non-primitive) positive hyperbolic orbit or a cover of a singular orbit. Moreover, if $\gamma_i$ is a cover of a singular orbit, then the monodromy along $\gamma_i$ does not rotate the stable-unstable quadrants.
\end{lem}

\begin{lem}\label{lem:samegrading}
    Let $g$ be a primitive element of $\pi_1(M)$. Then $\phi$ has a periodic orbit representing $g$ if and only if $(\omega_\phi, \lambda_\phi)$ has a Reeb orbit representing $g$. Moreover, after applying an exact stable homotopy if necessary, all Reeb orbits of $(\omega_\phi, \lambda_\phi)$ representing $g$ have the same Lefschetz index.
\end{lem}
\begin{proof}
    Recall that the Lefschetz index is +1 when the linearized return map along the orbit is elliptic or negative hyperbolic, and -1 when it is positive hyperbolic. For convenience, construct $(\omega_\phi,\lambda_\phi)$ using a Birkhoff section whose boundary orbits are not freely homotopic to any multiple of $[\gamma]$. Let $R$ be the Reeb flow of $(\omega_\phi, \lambda_\phi)$. The closed orbits of $\phi$ and $R$ are in 1-1 correspondence except at blownup orbits. When an orbit undergoes a clockwise or counterclockwise blowup, several positive hyperbolic orbits appear at the boundary of the eye, a single elliptic orbit appears at the centre, and many invariant Morse--Bott tori appear. When an orbit undergoes a nonrotating blowup, the single singular orbit splits into several positive hyperbolic orbits in the same free homotopy class and no other orbits appear.
\begin{enumerate}[label=Case \arabic*:]
    \item $\phi$ has no orbit representing $g$. Then neither does the Reeb flow of $(\omega_\phi, \lambda_\phi)$.
    \item $\phi$ has a negative hyperbolic or rotating singular orbit $\gamma$ representing $g$. By \cref{lem:lozenge}, $\gamma$ is the unique closed orbit of $\phi$ in its free homotopy class. To treat the two cases uniformly, just blow up the negative hyperbolic orbit. Since the monodromy along $\gamma$ is rotating, the positive hyperbolic closed orbits on the boundary of the resulting eye are not primitive. By a \ref{item:type4} exact stable homotopy, we can arrange that the slope of the flow inside the eye is very close to the slope of the hyperbolic orbits on the boundary of the eye. With this choice, none of these Morse--Bott closed orbits are primitive either. Therefore, the only closed Reeb orbit in $[\gamma]$ is the elliptic closed orbit at the centre of the eye.
    \item $\phi$ has a positive hyperbolic or nonrotating singular orbit $\gamma$ representing $g$. In this case, by \cref{lem:lozenge}, all the orbits of $\phi$ in the same free homotopy class are also positive hyperbolic or nonrotating singular. Therefore, all the Reeb orbits in this free homotopy class are positive hyperbolic.
\end{enumerate}

\end{proof}
\end{subsection}
\begin{subsection}{Rational symplectic field theory and free homotopy classes}
Rational symplectic field theory is the homology of a chain complex associated with a contact form. It is so named because it is a simplified version of symplectic field theory which counts only genus 0 holomorphic curves. The compactness results required for the construction of SFT continue to hold in the stable Hamiltonian setting; indeed, this was one of the original motivations for studying stable Hamiltonian structures \cite{bourgeois.eliashberg.ea.CompactnessResultsSymplectic, cieliebak.mohnke.CompactnessPuncturedHolomorphic}. In this paper, we do not address any of the foundational transversality issues in SFT, which are expected to have resolutions in any one of the virtual perturbation frameworks now available \cite{hofer.wysocki.ea.PolyfoldFredholmTheory, fish.hofer.LecturesPolyfoldsSymplectic}, \cite{pardon.AlgebraicApproachVirtual,pardon.ContactHomologyVirtual}, \cite{ishikawa.ConstructionGeneralSymplectic}. Instead, we focus on the adaptations necessary to formulate the theory for stable Hamiltonian structures. Compared to the contact case, two novelties arise:
\begin{fact}\label{fact:shorthomotopies}
    An exact stable homotopy between two SHS does not always give rise to a symplectic cobordism. A way around this is presented in \cite[Corollary 7.27]{cieliebak.volkov.FirstStepsStable}: an exact stable homotopy may be cut into short homotopies, each of which gives rise to a symplectic cobordism. In the stable Hamiltonian setting, one can compose two symplectic cobordisms provided that they are sufficiently ``short''. See \cite[Section 7]{cieliebak.volkov.FirstStepsStable} for more details, especially the definition of symplectic cobordism in Section 7.1 and the discussion leading up to Theorem 7.28.
\end{fact}
\begin{fact}\label{fact:positiveends}
    Holomorphic curves in the symplectization of an SHS may have no positive ends. This permits some new kinds of breaking of holomorphic curves. 
\end{fact}

\begin{rmk}
    
    Cylindrical contact homology is well defined for hypertight contact structures. An attempt to generalize the theory to stable Hamiltonian structures runs up against both \cref{fact:shorthomotopies} and \cref{fact:positiveends}. Because of \cref{fact:positiveends}, 1-parameter families of holomorphic cylinders can break as on the left side of \cref{fig:breaking}. Because of \cref{fact:shorthomotopies}, a homotopy between two hypertight SHS may not give rise to a symplectic cobordism. Instead, the homotopy must be cut into short homotopies which each give rise to a symplectic cobordism. However, we cannot guarantee that the stable Hamiltonian structures at which we cut are hypertight and we will need to consider breaking of holomorphic curves as on the right side of \cref{fig:breaking}.
\end{rmk}

\begin{figure}
    \centering
    \def\svgwidth{4in}
    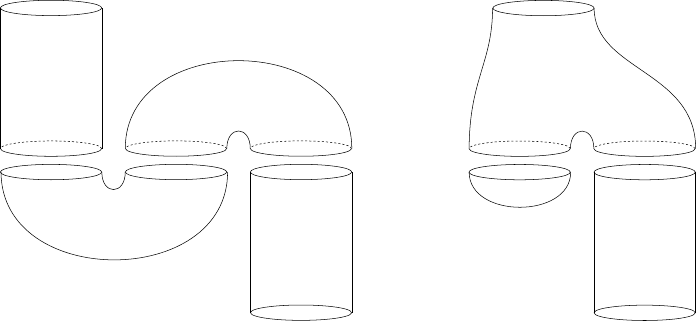
    \caption{Two ways a cylinder can break in a symplectization. In all our pictures, the symplectic cobordism goes from the top of the picture to the bottom of the picture.}\label{fig:breaking}
\end{figure}

Due to \cref{fact:positiveends}, the SFT generating function may now contain terms with negative powers of $\hbar$. Therefore, the original formulation of rational SFT as a limit $\hbar \to 0$ does not literally make sense. Instead, we use the ``$q$-variables only'' formulation proposed by Hutchings \cite{hutchings.RationalSFTUsing}. Here is a summary of the approach.

Let $(\omega,\lambda)$ be an SHS on $M$ with Morse--Bott Reeb flow. A closed, possibly non-primitive orbit $\gamma$ of $R$ is either positive hyperbolic, negative hyperbolic, elliptic, or part of an $S^1$ Morse--Bott family. For the purposes of SFT, each $S^1$ Morse--Bott family counts as two orbits, one positive hyperbolic and one elliptic. 2-dimensional Morse-Bott families can be eliminated by a small exact stable homotopy, so we may ignore them. Each orbit has a $\Z/2$ grading $gr$ depending on its Lefschetz sign: positive hyperbolic orbits have odd grading, while elliptic and negative hyperbolic orbits have even grading. An orbit is \emph{bad} if it is an even cover of a negative hyperbolic orbit; otherwise, it is \emph{good}.

Introduce variables $p_\gamma$, $q_\gamma$ for each good orbit of $R$, and define $gr(p_\gamma) = gr(q_\gamma)=gr(\gamma)$. Let $\mathcal{A}(\omega,\lambda)$ be the unital algebra over (the Novikov completion of) $\Q[H_2(M)]$ generated by all the variables $p_\gamma$, $q_\gamma$ modulo the graded commutativity relation $$xy = (-1)^{gr(x)  gr(y)} yx.$$ Choose a projection $A$ from 2-chains in $M$ to $H_2(M,\R)$. A holomorphic curve with positive ends at ${\gamma_1^+},\dots,{\gamma_j^+}$ and negative ends at $\gamma_1^-,\dots,\gamma_k^-$ is encoded by a monomial of the form $Ce^{A(u)}p_{\gamma_1^+}\dots p_{\gamma_k^+}q_{\gamma_1^-},\dots,q_{\gamma_k^-}$. The term $e^{A(u)}\in\Q[H_2(M)]$ keeps track of the relative homology class of $u$. We define the SFT generating function as $${\mathbf h}=\sum_{\gamma_i^+, \gamma_i^-}\sum_u C(u)p_{\gamma_1^+}\dots p_{\gamma_k^+}q_{\gamma_1^-},\dots,q_{\gamma_k^-}$$ where $C(u)\in \Q$ is a rational weight and the sum runs over all connected index 1, genus 0 holomorphic curves in a trivial cobordism $\R\times M$ asymptotic to $\gamma_1^+,\dots,\gamma_k^+$ on the positive end and $\gamma_1^-,\dots,\gamma_k^-$ on the negative end. We defer discussion of the rational weights to \cite{eliashberg.glvental.ea.IntroductionSymplecticField} or \cite[Lecture 12]{wendl.LecturesSymplecticField}.

Now we define a new algebra $\mathcal{A}'(\omega, \lambda)$ which will serve as the chain group for rational SFT. The generators of $\mathcal{A}'(\omega,\lambda)$ are monomials in the $q$-variables, equipped with a partition of the factors into $m\geq 0$ equivalence classes. Hutchings asks us to think of such a monomial as a collection of Reeb orbits, where Reeb orbits in the same equivalence class are boundary components of the same connected genus zero curve above the Reeb orbits, and so are not allowed to be glued to the same component of a holomorphic curve below the Reeb orbits.

Now we define a product $\mathcal A' \times \mathcal A \to \mathcal A'$. For monomials $x\in \mathcal A'$ and $y\in \mathcal A$, we are to interpret $x\cdot y$ as the sum of all the ways to glue the positive ends of a holomorphic curve $y$ to the collection of Reeb orbits $x$. Here is how we keep track of the relevant signs. Say that a $p$-variable matches a $q$-variable if they correspond to the same Reeb orbit. Form the product monomial $xy$, and choose for each $p$-variable in $y$ a matching $q$-variable in $x$, subject to the condition that all the matching $q$-variables come from different equivalence classes. This reflects the constraint that we must not connect two ends in the same equivalence class lest we create a curve of genus $>0$. Commute each $p$ variable to sit exactly to the right of its matching $q$ variable, and then annihilate the pair $q_\lambda p_\lambda$ and multiply by a rational weight $\kappa_\gamma$. Again, we defer discussion of the rational weights to \cite{eliashberg.glvental.ea.IntroductionSymplecticField} or \cite[Lecture 12]{wendl.LecturesSymplecticField} The result is a new monomial in $q$. The new $q$-variables from $y$ as well as all the equivalence classes of matched $q$ variables in $x$ are merged into a single equivalence class. Finally, $x\cdot y$ is the sum over all choices in this construction.

We define the differential $\partial$ on $\mathcal A'(\omega, \lambda)$ by $\partial x = x \cdot {\mathbf h}$. The fact that $\partial^2=0$ follows from looking at the ends of the moduli space of genus 0, index 2 holomorphic curves. Given a strong symplectic cobordism from  $(\omega_1,\lambda_1)$ to $(\omega_2, \lambda_2)$, one may define in a similar manner a cobordism map $$F:\mathcal A'(\omega_1,\lambda_1)\to \mathcal A'(\omega_2, \lambda_2)$$ which counts possibly disconnected index 0, genus 0 holomorphic curves in the completed cobordism. Finally, there are chain homotopies $$K:\mathcal A'_*(\omega_1,\lambda_1)\to \mathcal A'_{*+1}(\omega_2, \lambda_2)$$ associated to homotopies of symplectic cobordisms, defined by counting possibly-disconnected index -1, genus 0 curves. As a consequence of the algebraic setup, the composition of any combination of these maps counts only genus 0 curves.

\begin{lem}\label{lem:sameorbits}
    Let $g$ be a primitive element of $\pi_1(M)$. Suppose there is a Reeb orbit $\gamma$ of $(\omega,\lambda)$ representing $g$, and all other such Reeb orbits have the same Lefschetz index. Suppose $(\omega, \lambda)$ is exact stable homotopic to $(\omega',\lambda')$.  Then there is at least one orbit of $(\omega', \lambda')$ representing $g$.
\end{lem}
\begin{proof}
    Since $g$ is primitive, $\gamma$ is good. By \cite[Corollary 7.27]{cieliebak.volkov.FirstStepsStable}, there is a sequence of symplectic cobordisms $X_i$ and $Y_i$ from $(\omega,\lambda)$ to $(\omega',\lambda')$
    $$(\omega, \lambda)=(\omega_0,\lambda_0)
    \xrightleftharpoons[Y_1]{X_1}
    (\omega_1,\lambda_1)
    \xrightleftharpoons[Y_2]{X_2}\dots\xrightleftharpoons[Y_n]{X_n}
    (\omega_n,\lambda_n)=(\omega',\lambda')$$ such that $X_i \circ Y_i$ and $Y_i \circ X_i$ are both homotopic to trivial symplectic cobordisms. Let $F_i$ be the cobordism map associated to $X_i$, let $G_i$ be the cobordism map associated to $Y_i$, and let $K_i$ be the chain homotopy between $G_i \circ F_i$ and the identity. Then $$G_1 \circ \dots \circ G_n \circ F_n \circ \dots \circ F_1:\mathcal A'(\omega, \lambda) \to \mathcal A'(\omega, \lambda)$$ is chain homotopic to the identity via the chain homotopy
    $$K:=\sum_{i=1}^n G_1 \circ \dots G_{i-1} \circ K_i \circ F_{i-1} \circ \dots \circ F_1$$

    The specific form of $K$ is unimportant. We only need to know that it counts broken genus 0 curves in a sequence of topologically trivial symplectic cobordisms, and that \begin{equation}\label{eqn:homotopy}G_1\circ\dots\circ G_n \circ F_n \circ \dots \circ F_1-id=K\partial + \partial K.\end{equation} Let us focus on broken holomorphic curves from $\gamma$ to $\gamma$ contributing to any of the terms in \cref{eqn:homotopy}. We claim that there are no such curves counted by $K\partial + \partial K$. To distinguish the two copies of $\gamma$, call the one at the positive end $\gamma^+$ and the other $\gamma^-$.  Suppose for example that there is a broken curve $(u_1, u_2,\dots,u_k)$ from $\gamma^+$ to $\gamma^-$ contributing to $\partial K$, where $(u_1,\dots,u_{k-1})$ contributes to $K$ and $u_k$ contributes to $\partial$. The positive end of $u_{k}$ may have several components, but since $u_1 \cup u_2\cup \dots \cup u_k$ has genus 0, one of those components separates $\gamma^+$ from $\gamma^-$. Call that orbit $\mu$. In fact, there cannot be any other positive ends of $u_{k}$, because those ends would have to bound disks but $(\omega, \lambda)$ has no contractible orbits. This rules out configurations like the one shown in \cref{fig:breaking2}. Therefore, $u_k$ is an index 1 cylinder from $\mu$ to $\gamma^-$ counted by $\partial$. But this contradicts \cref{lem:samegrading} which says that all the representatives of $[\gamma]$ have the same grading. In the same way, one can rule out curves contributing to $K\partial$.
    
    \begin{figure}
        \centering
        \def\svgwidth{3.5in}
        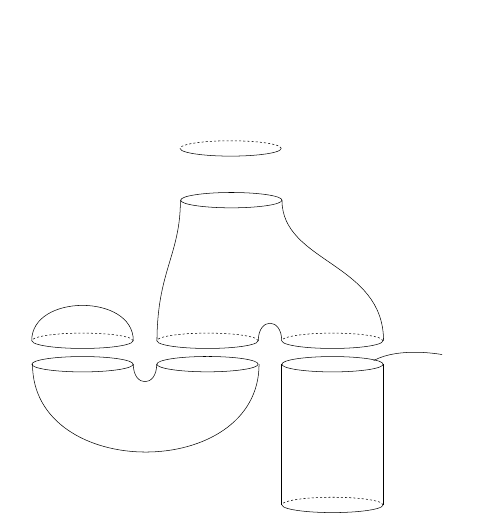
        \caption{}\label{fig:breaking2}
    \end{figure}

    It follows that there is a (possibly disconnected) genus 0 broken curve $u$ counted by $G_1 \circ \dots \circ G_n \circ F_n \circ \dots \circ F_1$ from $\gamma^+$ to $\gamma^-$. Since $\gamma$ is not contractible, some component of $u$ must connect $\gamma^+$ to $\gamma^-$. As before, some Reeb orbit $\mu$ of $(\omega',\lambda')$ must separate $\gamma^+$ from $\gamma^-$. All the remaining Reeb orbits along which $u$ breaks at that level bound disks in $u_1\cup u_2 \cup \dots u_k$. Therefore, $\mu$ is a Reeb orbit of $(\omega', \lambda')$ freely homotopic to $\gamma^+$.
\end{proof}

\begin{proof}[Proof of \cref{thm:main}] Suppose $(\omega_{\phi_1},\lambda_{\phi_1})$ is exact stable homotopic to $(\omega_{\phi_2},\lambda_{\phi_2})$. Call their Reeb flows $R_1$ and $R_2$. For any primitive element $g$ of $\pi_1(M)$, use \cref{lem:samegrading} to arrange that $\phi_i$ has an orbit representing $g$ if and only if $R_i$ has an orbit representing $g$, and moreover that all such orbits have the same Lefschetz sign. By \cref{lem:sameorbits}, $R_1$ has a Reeb orbit representing $g$ and only if $R_2$ does. Therefore, the primitive elements of $\pi_1(M)$ represented by $\phi_1$ are the same as those represented by $\phi_2$. The theorem now follows from \cref{thm:reconstruction} and \cref{lem:primitive}.
\end{proof}

\begin{rmk}
    It is sometimes possible to rule out genus 0 holomorphic curves for topological reasons. Given an $\R$-covered foliation $\mathcal F$ on a hyperbolic 3-manifold, Fenley and Calegari constructed a pseudo-Anosov flow transverse to $\mathcal F$ which \emph{regulates} $\mathcal F$. The regulating condition implies that each element of $\pi_1(M)$ represented by a closed orbit acts as an increasing homeomorphism on the leaf space of $\mathcal F$. It follows that no positive product of such elements is the identity. Therefore, there are no genus 0 holomorphic curves with only positive or only negative ends. This means that the unit in $\mathcal A'$, which corresponds with the empty set of Reeb orbits, is a nontrivial element of the homology of rational SFT. These stable Hamiltonian structures ought to be called tighter than hypertight, but we are quickly running out of superlative prefixes.
\end{rmk}

\begin{section}{Transverse contact structures}
    In preparation for proving \Cref{prop:tight}, we observe a slightly more general statement which applies to any hypertight stable Hamiltonian structure.
\begin{prop}\label{prop:transversecontact}
    Suppose $(\omega, \lambda)$ is a hypertight stable Hamiltonian structure on an irreducible rational homology 3-sphere $M$ and $\xi$ is a contact structure with $\omega|_\xi > 0$. Then $\xi$ is tight and has zero Giroux torsion.
\end{prop}
\begin{proof}
    Consider the symplectic manifold $[0,\varepsilon] \times M$ with the symplectic form $\Omega = \omega + d(t\lambda)$. This is a symplectic cobordism from $(\omega,\lambda)$ to $\xi$ which is strong on the end $(\omega, \lambda)$ and weak on the $\xi$ end. Attach a cylindrical end onto $0 \times M$. Choose an asymptotically cylindrical almost-complex structure compatible with $\Omega$. Now we can imitate Gromov and Eliashberg's argument that a weakly fillable contact structure is tight \cite{gromov.PseudoHolomorphicCurves,eliashberg.FillingHolomorphicDiscs}. If $\xi$ has an overtwisted disk $D$, then we can choose $J$ with a standard form near $\varepsilon \times D$ such that there is a 1-parameter family of holomorphic disks with boundary on $D$. This family is usually called the Bishop family. The family has one endpoint at the centre of $D$. By SFT compactness, it must have another endpoint, but we will rule out all the possibilities for this other endpoint. A sequence of holomorphic disks cannot exit $\varepsilon \times M$ since that end is pseudo-convex. A sequence of holomorphic disks cannot break at the cylindrical end attached to $0\times M$, since the Reeb orbits along which it breaks must be contractible, but $(\omega,\lambda)$ has no contractible Reeb orbits. Bubbling is prohibited since $M$ is irreducible. Contradiction. Thus, $\xi$ must be tight.

    Now suppose that $\xi$ has nonzero Giroux torsion. Since $M$ is a rational homology 3-sphere, $\omega + \varepsilon d\lambda$ is exact on $\varepsilon \times M$. Eliashberg showed that a weak exact filling may be modified in a collar neighbourhood of the boundary so that it becomes a strong filling \cite[Proposition 13.19]{wendl.LecturesSymplecticField}, \cite[Prop 3.1]{eliashberg.SymplecticManifoldsContact}. Thus, we have a strong symplectic cobordism $X_1$ from $(\omega,\lambda)$ to $\xi$. If $\xi$ has Giroux torsion, then \cite[Theorem 1]{wendl.NonexactSymplecticCobordisms} states that there is a strong symplectic cobordism $X_2$ from $\xi$ to an overtwisted contact structure. The composition of $X_1$ and $X_2$ is a strong symplectic cobordism from $(\omega,\lambda)$ to an overtwisted contact structure. We get a contradiction by applying the argument from the previous paragraph.
\end{proof}

\begin{proof}[Proof of \cref{prop:tight}]
    Suppose $\phi$ has a positive Birkhoff section. Let $\mathcal B$ be the associated open book and let $\alpha$ be a contact form for the Thurston--Winkelnkemper contact structure. Choose small enough blowup parameters that the Reeb flow of $(\omega,\lambda)$ is transverse to $\ker(\alpha)$ and supported by $\mathcal B$. By \cref{prop:transversecontact}, $\ker(\alpha)$ is tight and has zero Giroux torsion. While the Reeb flow of $(\omega,\lambda)$ may not be the Reeb flow of a contact structure, it is compatible with $\mathcal B$. Now $(\omega,\lambda)$ and $(d\alpha, \alpha)$ are supported by the same open book. Therefore, by \cref{prop:openbook}, they are possibly-non-exact stable homotopic.
\end{proof}

\end{section}

\begin{section}{Dehn surgery}
    In this section, we prove \cref{thm:surgery}. Recall the set $$V=\{p/q \mid \gcd(p,q)=1,\, p\neq 0,\,q \geq 6,\text{ and } \,p \neq 0, \pm 1, \pm 2 \mmod q\}$$ from the statement of \cref{thm:surgery}. The significance of this set is explained by the first lemma:

    \begin{lem}\label{lem:hyperbolicsurgery}
        Suppose $p/q\in V$ and $K$ is a hyperbolic knot in $S^3$. Then $S^3_{p/q}(K)$ admits a pseudo-Anosov flow $\phi$ such that the Dehn surgery core is a singular orbit. Moreover, $\phi$ has either a positive or negative Birkhoff section.
    \end{lem}
    \begin{proof}

        Since $K$ is hyperbolic, the fibration of $S^3\setminus K$ admits a transverse pseudo-Anosov flow $\phi^\circ$. The \emph{degeneracy slope}, which we call $r/s$, is the slope of the prongs of $\phi^\circ$ on the cusp of $S^3 \setminus K$. Let $\iota(x,y)$ denote the intersection number between two slopes on $T^2$. If $\iota(r/s,p/q)\geq 2$, then $\phi^\circ$ blows down to a pseudo-Anosov flow $\phi$ on $S^3_{p/q}(K)$. Moreover, the Dehn surgery core is a closed orbit with $\iota(r/s,p/q)$ prongs. Since $S^3$ does not admit a pseudo-Anosov flow, we know that the meridian intersects the degeneracy slope zero or one times. It follows that $r/s$ is either $1/0$ or $r/1$. In our two cases, we have
        \begin{align*}
            \iota(1/0, p/q) &= |q|\\
            \iota(r/1, p/q) &= |p-rq|
        \end{align*}
        We defined $V$ so that both $|q|$ and $|p-rq|$ are at least 3 whenever $p/q \in V$. Therefore, the Dehn surgery core is a singular orbit of $\phi$. The fibers of the fibration of $S^3\setminus K$ descend to either positive or negative Birkhoff sections of $\phi$.

        
    \end{proof}

    \begin{proof}[Proof of \cref{thm:surgery}]
        Let $M$ be a closed, oriented manifold. If $M$ is obtained as $p/q$ surgery along a fibered hyperbolic knot $K$ for some slope $p/q \in V$, then \cref{lem:hyperbolicsurgery} tells us that the Dehn surgery core in $S^3_{p/q}(K)$ is a singular orbit of a transitive pseudo-Anosov flow $\phi$ on $M$. Moreover, $\phi$ has either a positive or a negative Birkhoff section. \Cref{thm:contact_finiteness} tells us that there are finitely many such pseudo-Anosov flows on $M$. Each of those pseudo-Anosov flows has finitely many singular orbits. Thus, there are only finitely many candidates for the Dehn surgery core. By the Gordon--Luecke theorem, for each of those singular orbits, there is at most one surgery coefficient which gives us $S^3$. Therefore, there are finitely many possibilities for $K$.
    \end{proof}



\end{section}

\begin{section}{Questions}

    \begin{ques}Is there a Nielsen--Thurston classification for stable Hamiltonian structures? As a first step, is there a hypertight SHS on a hyperbolic 3-manifold without a pseudo-Anosov representative up to possibly-non-exact stable homotopy?
    \end{ques}
    \begin{ques}\label{ques:ET}
        Given a taut foliation $\F$ on a closed, oriented 3-manifold not homeomorphic to $S^1\times S^2$, the Eliashberg--Thurston theorem produces a positive and negative contact structure $\xi_+$ and $\xi_-$ approximating $T\F$. When do these two contact structures lie in the same stable Hamiltonian class? When do they admit pseudo-Anosov representatives?
    \end{ques}
When $M$ is a rational homology sphere and $(\omega,\lambda)$ is an SHS with Reeb flow transverse to $\F$, there are always strong symplectic cobordisms from $\xi_-$ to $(\omega,\lambda)$ and from $(\omega,\lambda)$ to $\xi_+$. The main question is whether there are symplectic cobordisms in the other direction.
    \begin{ques}
        Does \cref{thm:main} hold when the condition of exactness is dropped? See \cref{rmk:nonexact}.
    \end{ques}
    \begin{ques}
        If a transitive pseudo-Anosov flow $\phi$ has a positive Birkhoff section, then $(\omega_\phi,\lambda_\phi)$ is possibly-non-exact stable homotopic to a positive contact structure. Is the converse true?
    \end{ques}

\end{section}



\end{subsection}
\end{section}

\printbibliography
\end{document}

%% file: breaking.pdf_tex
\begingroup%
  \makeatletter%
  \providecommand\color[2][]{%
    \errmessage{(Inkscape) Color is used for the text in Inkscape, but the package 'color.sty' is not loaded}%
    \renewcommand\color[2][]{}%
  }%
  \providecommand\transparent[1]{%
    \errmessage{(Inkscape) Transparency is used (non-zero) for the text in Inkscape, but the package 'transparent.sty' is not loaded}%
    \renewcommand\transparent[1]{}%
  }%
  \providecommand\rotatebox[2]{#2}%
  \newcommand*\fsize{\dimexpr\f@size pt\relax}%
  \newcommand*\lineheight[1]{\fontsize{\fsize}{#1\fsize}\selectfont}%
  \ifx\svgwidth\undefined%
    \setlength{\unitlength}{334.05856705bp}%
    \ifx\svgscale\undefined%
      \relax%
    \else%
      \setlength{\unitlength}{\unitlength * \real{\svgscale}}%
    \fi%
  \else%
    \setlength{\unitlength}{\svgwidth}%
  \fi%
  \global\let\svgwidth\undefined%
  \global\let\svgscale\undefined%
  \makeatother%
  \begin{picture}(1,0.4608878)%
    \lineheight{1}%
    \setlength\tabcolsep{0pt}%
    \put(0,0){\includegraphics[width=\unitlength,page=1]{breaking.pdf}}%
  \end{picture}%
\endgroup%

%% file: breaking2.pdf_tex
\begingroup%
  \makeatletter%
  \providecommand\color[2][]{%
    \errmessage{(Inkscape) Color is used for the text in Inkscape, but the package 'color.sty' is not loaded}%
    \renewcommand\color[2][]{}%
  }%
  \providecommand\transparent[1]{%
    \errmessage{(Inkscape) Transparency is used (non-zero) for the text in Inkscape, but the package 'transparent.sty' is not loaded}%
    \renewcommand\transparent[1]{}%
  }%
  \providecommand\rotatebox[2]{#2}%
  \newcommand*\fsize{\dimexpr\f@size pt\relax}%
  \newcommand*\lineheight[1]{\fontsize{\fsize}{#1\fsize}\selectfont}%
  \ifx\svgwidth\undefined%
    \setlength{\unitlength}{238.11779594bp}%
    \ifx\svgscale\undefined%
      \relax%
    \else%
      \setlength{\unitlength}{\unitlength * \real{\svgscale}}%
    \fi%
  \else%
    \setlength{\unitlength}{\svgwidth}%
  \fi%
  \global\let\svgwidth\undefined%
  \global\let\svgscale\undefined%
  \makeatother%
  \begin{picture}(1,1.03381931)%
    \lineheight{1}%
    \setlength\tabcolsep{0pt}%
    \put(0,0){\includegraphics[width=\unitlength,page=1]{breaking2.pdf}}%
    \put(0.89527297,0.30677278){\makebox(0,0)[lt]{\lineheight{1.25}\smash{\begin{tabular}[t]{l}$\mu$\end{tabular}}}}%
    \put(0.68353911,0.97899038){\makebox(0,0)[lt]{\lineheight{1.25}\smash{\begin{tabular}[t]{l}$\gamma^+$\end{tabular}}}}%
    \put(0.21719774,0.5547433){\makebox(0,0)[rt]{\lineheight{1.25}\smash{\begin{tabular}[t]{r}$u_{k-1}$\end{tabular}}}}%
    \put(0.40498023,0.04870926){\makebox(0,0)[rt]{\lineheight{1.25}\smash{\begin{tabular}[t]{r}$u_k$\end{tabular}}}}%
    \put(0,0){\includegraphics[width=\unitlength,page=2]{breaking2.pdf}}%
    \put(0.46575116,0.67340341){\makebox(0,0)[t]{\lineheight{1.25}\smash{\begin{tabular}[t]{c}$\vdots$\end{tabular}}}}%
    \put(0.87679725,0.01889181){\makebox(0,0)[lt]{\lineheight{1.25}\smash{\begin{tabular}[t]{l}$\gamma^-$\end{tabular}}}}%
    \put(0,0){\includegraphics[width=\unitlength,page=3]{breaking2.pdf}}%
    \put(0.23459565,0.86944593){\makebox(0,0)[rt]{\lineheight{1.25}\smash{\begin{tabular}[t]{r}$u_1$\end{tabular}}}}%
    \put(0,0){\includegraphics[width=\unitlength,page=4]{breaking2.pdf}}%
  \end{picture}%
\endgroup%